\documentclass[12pt]{article}

\font\naam=cmssbx10 scaled \magstep1
\font\kop=cmss10 scaled \magstep1

\usepackage{amsmath, amsthm, amsfonts}
\DeclareMathSymbol{\subsetneq}{\mathord}{AMSb}{"26}

\newtheorem{lemma}{Lemma}[section]

\newtheorem{theorem}[lemma]{Theorem}

\newtheorem{proposition}[lemma]{Proposition}
\newtheorem{corollary}[lemma]{Corollary}
\theoremstyle{definition}

\newtheorem{definition}[lemma]{Definition}

\newtheorem{remark}[lemma]{Remark}

\newcommand{\lp}{\longrightarrow}

\newcommand{\mb}{\mathbb}

\newcommand{\C}{\mb{C}}

\newcommand{\Z}{\mb{Z}}
\newcommand{\N}{\mb{N}}

\begin{document}

\font\ti=cmssbx10 scaled \magstep2
\font\naam=cmssbx10 scaled \magstep1
\font\kop=cmss10 scaled \magstep1

\newpage
\setcounter{page}{1}

\title{$\C$-flows $A^z$ of linear maps $A$ expressed in terms of  $A^{-1},A^{-2},\ldots,A^{-n}$ and analytic functions of $z$.}
\author{Stefan Maubach\footnote{Funded by Veni-grant  from the Dutch Organization for Scientific Research (NWO)}\\ \ \\
\small
Radboud University Nijmegen\\\small Toernooiveld 1, The Netherlands\\ \small s.maubach@math.ru.nl}

\maketitle

\abstract{Suppose $A\in GL_n(\C)$ has a relation $A^p=c_{p-1}A^{p-1}+\ldots + c_1 A+ c_0I $ where the $c_i\in\C$.
This article describes how to construct analytic functions $c_i(z)$ such that $A^z=c_{p-1}(z)A^{p-1}+\ldots + c_1(z) A+ c_0(z)I $. 
One of the theorems gives a possible description of the $c_i(z)$: $c_i(z)=C^z\alpha$ where
$C\in Mat_p(\C)$ is (similar to) the companion matrix of $X^p-c_{p-1}X^{p-1}-\ldots -c_1X-c_0I$, and $\alpha:=
(c_{p-1},\ldots,c_1,c_0)^t$.}

\normalsize

\section{Introduction}

\subsection{Motivation}

Given an invertible linear map on a finite dimensional $\C$-vector space (which we will view in this article as a square matrix $A$), 
then it is possible to define ``$A^z$'' for each $z\in \C$ in a ``good way''. Let us define what we mean:

\begin{definition}
Given an invertible  map $f:\C^n\lp\C^n$,
we say that the map $\varphi: \C\times \C^n\lp \C^n$
is a $\C$-flow (sometimes ``exponent map'') of $f$ if it satisfies\\
(1) $\varphi(z,-)\varphi(w,-)=\varphi(z+w,-)$ for all $z,w\in \C$,\\
(2) $\varphi(1,-)= f$,\\
(3) $\varphi(0,-)=I$.
\end{definition}

So, $A^z$ is defined in a good way if it is $\C$-flow of $A$.
This is easily done if $A$ is on Jordan-normal form (or on upper triangular form), and then by conjugation one can do it for any invertible matrix. 

In this article we will describe how to find a $\C$-flow of $A$, given a relation $A^p=c_{p-1}A^{p-1}+\ldots + c_1 A+ c_0I $ where the $c_i\in\C$.
We do this by giving a construction of analytic functions $c_i(z)$ that only depend on the $c_i$, such that $A^z$ can be defined as 
$c_{p-1}(z)A^{p-1}+\ldots + c_1(z) A+ c_0(z)I$. 
This may be more efficient in the case the matrix $A$ is of significantly larger size than $p$; 
then it is  cheap to compute the functions $c_i(z)$ and $A^i$ where $0\leq i \leq p-1$, and  expensive to compute a formula $A^z$ by a direct method. 

Next to this purely linear algebra motivation, there is another one that originates in a  problem not concerning linear maps, 
which we will only quote as a motivation. It concerns the more general polynomial maps  
$F:\C^n\lp \C^n$ that satisfy relations $F^n=c_{n-1}F^{n-1}+\ldots c_1F+c_0I$ where $c_i\in \C$.
It was conjectured in \cite{A} (or see \cite{B} paragraph 4.3)that such maps are an exponent of a so-called locally finite derivation.
Equivalently, it was enough to show that for such a polynomial map there exists
an exponent map, i.e. a set of polynomial maps $\{F_t~|~t\in \C\}$ such that $F_0=I, F_1=F,$ and $F_tF_u=F_{t+u}$. 
The author found a surprising formula that gave this exponent map in all concrete cases:\\
\[
F_t:=(F^{-1},F^{-2},\ldots,F^{-n})
\left(
\begin{array}{ccccc}
c_{n-1} & 1 & 0 & \cdots &0\\
c_{n-2} & 0 & 1 & \cdots & 0\\
\hdots & \vdots & \vdots&\ddots &\hdots\\
c_1 & 0 & 0 & \cdots&1\\
c_0 & 0 & 0 & \cdots &0\\
\end{array}\right)^t
\left(
\begin{array}{c}
c_{n-1} \\ \hdots \\ c_1 \\ c_0
\end{array}
\right).\]
However, the author was not able to prove that this formula always worked in general. 
But, the fact that this formula works for linear maps is exactly this article. 

We will formulate our main results in section 2. 
In order to prove our results, it is necessary to recall and notice some basic linear algebra facts, which we do in section 3. 
In the same section we will discuss basic facts on constructing the map $z\lp A^z$ in a well-defined way.  
In section 4 we will briefly discuss the set $V_A:=\oplus_{i\in \N} \C A^i$.
This will be used in section 5 to prove the  main result. In section 6 we will discuss how one could compute the analytic functions in practice.

\subsection{Notations}

$Mat_n(\C)$ will denote the set of $n\times n$ matrices, where $GL_n(\C)$ will denote the subset of $Mat_n(\C)$ of invertible matrices. 
$I_n$ will denote the identity map in $GL_n(\C)$.
$N_n$ will denote the map in $Mat_n(\C)$ which is zero everywhere except at the $n-1$ positions directly above the diagonal, 
where it is 1. If $n$ is known, it may be omitted, writing $I$ and $N$.

We will say that two matrices $A,B\in Mat_n(\C)$ are conjugate, if there exists $T\in GL_n(\C)$ such that $A=T^{-1}BT$. In particular, 
a matrix is always conjugate to one of its
Jordan normal forms. 

If $Q(T):=T^p-c_{p-1}T^{p-1}-c_{p-2}T^{p-2}-\ldots-c_1T-c_0$, then we define the {\em companion matrix} (though usually defined slightly different) of $Q$ as
\[
C_Q:=\left(
\begin{array}{ccccc}
c_0&1&0&\ldots&0\\
c_1&0&1&\ldots&0\\
&\vdots&    &\vdots\\
c_{p-1}&0&0&\ldots&1\\
c_p&0&0&\ldots&0\\
\end{array}\right).
\]
If $A\in Mat_n(\C)$, and $Q(T)$ is the characteristic polynomial of $A$, then we define the companion matrix of $A$ as $C_A:=C_Q$.

\section{Main results}

We pull the main results from this article to the front, and leave the proofs for later. 
In the below result, we assume that we have  a $\C$-flow $C_Q^z$.

\begin{theorem}\label{2005.5.26.C}
Let $A\in GL_n(\C)$, and $Q(X):=X^p-c_{p-1}X^{p-1}-\ldots -c_1X-c_0\in\C[X]$ such that $Q(A)=0$.
Let $c=(c_{p-1},\ldots,c_0)^t$, and $\mu_i(z):=(C_Q^zc)$. Then $A^z=\sum_{i=1}^p \mu_i(z) A^{-i}$.
\end{theorem}

The proof of this theorem is in section 5.

Below we give a method to compute the functions $c_i(z)$ in a more direct way. This does not need a given $\C$-flow $C_Q^z$. 
The following is a concise summary of the results of section 6.

Suppose a matrix $A$ satisfies a relation
\[ A^p-c_{p-1}A^{p-1}-\ldots -c_1A-c_0I=0. \]
Then one can define a formula for $A^z$ in the following way:\\
\begin{itemize}
\item Compute the zeroes of $X^p-c_{p-1}X^{p-1}-\ldots -c_1X-c_0\in\C[X]$ and count their multiplicity.
Name the zeroes $\lambda_1,\ldots,\lambda_m$ with multiplicities $n_1,\ldots,n_m$.
\item Choose for each $1\leq i \leq m$ a function $\lambda_i^z$ (which should be an exponent map, i.e. $\C$-flow).
\item Define for each $1\leq i \leq m$, $0\leq j \leq n_i-1$ the functions $h_{ij}:=\lambda^{z-j} g_j(z)$
where $g_0(z)=1, g_j(z)=\frac{1}{j!}z(z-1)\cdots (z-j+1)$ if $j\geq 1$. Rename the $p$ functions $h_{ij}$ obtained in this way as $f_1(z),\ldots, f_p(z)$.
\item Compute the inverse of $(f_i(-j))_{ij}$ where $i,j$ run from 1 to $p$, and name the inverse $e_{ij}$.
\item Now
\[ A^z = \sum_{i=1}^p \sum_{j=1}^p e_{ij}f_j(z) A^{-i} .\]
\end{itemize}

\section{Preliminaries}

\subsection{Jordan normal form and minimum polynomial}

The remarks \ref{04.4.F} and \ref{04.4.L} are standard, and do not have difficult proofs, which are omitted.

\begin{remark}\label{04.4.F}
We have equivalence between
\begin{enumerate}
\renewcommand{\theenumi}{(\roman{enumi})}
\item $A\in GL_n(\C)$ has minimum polynomial of degree $n$;
\item The $m$ blocks of the Jordan-normal-form of $A\in GL_n(\C)$ have eigenvalues $\lambda_1,\ldots,\lambda_m$
which all differ;
\item $A\in GL_n(\C)$ is a conjugate of its companion matrix $C_A\in GL_n(\C)$.
\end{enumerate}
\end{remark}

\begin{remark}\label{04.4.L}
Let $m(X)$ be the minimum polynomial of $A\in Mat_n(\C)$ of degree $p$.
Then there exist $J\in Mat_n(\C), \tilde{J}\in Mat_p(\C)$, such that\\
\begin{enumerate}
\item
$J$ is a
Jordan normal form of $A$, and $\tilde{J}$ is a Jordan normal form of $C_m$,
\item
$\tilde{J}$ is the upper left $p\times p$ minor of $J$, consisting of a subset of the blocks appearing in $J$,
\item
per eigenvalue $\lambda$ occurring in $A$ (or $C_m$), there is only one block having this eigenvalue in
$\tilde{J}$,  and it has the maximum of the size of all the blocks having eigenvalue $\lambda$ in $A$.
\end{enumerate}
\end{remark}

\subsection{$\C$-flows of linear maps}

It is well-known how to make well-defined $\C$-flow $\C\times \C^n\lp \C^n:(z,v)\lp A^zv$ for a given $A\in GL_n(\C)$.
This section provides some of these details, and points out a pitfall which one should keep in
mind: the freedom in choosing {\em which} $\C$-flow.

Note that a $\C$-flow endorses the notation $f^z:=\varphi(z,-)$,
as this coincides with the normal notation $f^n=f\circ f\circ\cdots\circ f$ when $z=n\in\N$,
and $\varphi(-1,-)$ indeed is the same map as $f^{-1}$ since it is the inverse of $\varphi(1,-)=f$.
To take a trivial example, let $f=I$, and take $f^z=I$.
To take a little less trivial example for the same $f$, take $f^z:=e^{Im(z)} I$
(and $\varphi(z,x)=f^z(x)=e^{Im(z)}x$).
This shows one of the problems in finding $\C$-flows: they are not unique.
For some specific functions, like the map $x\lp ex$, a standard choice has been made by literature ($\varphi(z,x)\lp e^zx$).
But if one wants to define an exponent of the map $x\lp \lambda x$ where $\lambda\in \C^*$, then
one has to fix a value $\log(\lambda)$ such that $e^{\log(\lambda)}=\lambda$,
and then define $\varphi(z,x):=e^{z\log(\lambda)}x$.

In fact, in this article, we assume that we have {\bf fixed} for every $\lambda\in\C$  one such value
$\log(\lambda)$ and a corresponding map  $z\lp \lambda^z$.
Let us elaborate shortly on  how this lets us fix maps $z\lp A^z$ for each $A\in GL_n(\C)$, even though this is mainly standard technique,  
as we will need some of the details later.

\begin{definition}\label{05.4.A}
Define $g_0(z)=1$, $g_i(z):=\frac{1}{i!}z(z-1)\cdots (z-i+1)$ for all $i\in \N^*$.
Note that if $k\in \N$, then $g_i(k)= {k\choose i}$ (even if $k<i$).\\
Given $\lambda\in\C^*$, define $B_{n}(\lambda):=\lambda I + N\in GL_n(\C)$.
\end{definition}

\begin{lemma}\label{05.4.B}
Let  $\varphi(z,-):=\sum_{i=0}^n g_i(z) \lambda^{z-i} N^i$.
Then $\varphi$ is a $\C$-flow of $B_n(\lambda)$.
\end{lemma}

\begin{proof}
{\em (i)} Define
$P_k(z,w):=\sum_{l=0}^k g_l(z)g_{k-l}(w)-g_k(z+w)=0$ for all $z,w\in\C, k\in N$.
Note that $P_k(z,w)$ is a polynomial. If $z,w\in \N$ then $P_k(z,w)$ equals
$\sum_{l=0}^k {z \choose l}{w\choose k-l} - {z+w\choose k}$, which in turn is equal to
``The coefficient of $X^k$ in  $(1+X)^z(1+X)^w$ minus the coefficient of $X^k$ in $(1+X)^{z+w}$,
which is obviously equal to zero.
Thus, $P_k(z,w)=0$ for all $z,w\in\N$, but for a polynomial this is only possible if the polynomial is zero.
Thus $P_k(z,w)=0$ for all $z,w\in\C$, and therefore $\sum_{l=0}^k g_l(z)g_{k-l}(w)=g_k(z+w)$.\\
{\em (ii)}
$\varphi(0,-)=\sum_{i=0}^n g_i(0) \lambda^{-i} N^i=I$
and $\varphi(1,-)=\sum_{i=0}^n g_i(1) \lambda^{1-i} N^i=\lambda^{1-0}I+g_1(1)\lambda^{1-1} N=
\lambda I+N=B$, therefore only left to prove is:
$\varphi(z,-)\varphi(w,-)=\varphi(z+w,-)$.
Let us do some computation:
\[\begin{array}{rl}
\varphi(z,-)\varphi(w,-)=&\big(\sum_{i=0}^n g_i(z) \lambda^{z-i} N^i\big)\big( \sum_{i=0}^n g_i(w) \lambda^{w-i} N^i\big)\\
=&\sum_{k=0}^n\sum_{l=0}^k g_l(z)\lambda^{z-l}N^l g_{k-l}(w)\lambda^{w-k+l} N^{k-l}\\
=&\sum_{k=0}^n \big( \sum_{l=0}^k g_l(z)g_{k-l}(w)\big) \lambda^{z+w-k}N^k\\
=^{(i)}&\sum_{k=0}^n g_k(z+w) \lambda^{z+w-k} N^k\\
=& \varphi(z+w,-).
\end{array} \]
\end{proof}

\begin{remark}\label{04.4.E}
The notation $B^z$ is now justified. Note that one may see $B^z$ as an element in $GL_n(H(z))$ where $H(z)$ is the set of holomorphic functions on $\C$. The only nonzero coefficients appearing in $B^z$ are the functions $g_i(z)\lambda^{z-i}$ for $i=0,\ldots,n-1$.
\end{remark}

Now, let  $J\in GL_n(\C)$ be in Jordan normal form, having blocks $B_1,\ldots, B_k$,
where each $B_i=\lambda_i I_{n_i}+ N_{n_i}$ for some $n_i\in\N$ (and $\sum_{i=1}^m n_i=n$).
Now one can easily define $J^z$ by exponentiating each component $B_i$.

\begin{definition}
Let $A\in GL_n(\C)$, and $J$ a Jordan normal form of $A$ (i.e. $M^{-1}AM=J$ for some $M\in GL_n(\C)$).
Then define $A^z:=MJ^zM^{-1}$.
\end{definition}

The definition is well-defined as the choice of Jordan-normal form is irrelevant (it will have the same blocks, only permuted).

\begin{remark}\label{04.4.M}
If one chooses different exponential maps $\lambda^z$, then one obtains different $\C$-flows maps
$A^z$, but when the maps $\lambda^z$ are fixed, the maps $A^z$ are also fixed (if one uses the above construction).\\
Also, the only nonzero coefficients appearing in $J^z$ are the only nonzero coefficients appearing in the
$B_j^z$ : $g_i(z)\lambda_j^{z-i}$ for $j=1,\ldots,m, i=0,\ldots, n_i-1$.
\end{remark}

\section{The ring $V_A$}

\begin{definition}
If $A\in GL_n(\C)$, then define $V_A$ as the linear subspace of $Mat_n(\C)$ generated by $\{A^i~|~i\in \Z\}$. Written differently: 
$V_A:=\oplus_{i\in\Z} \C A^i$. Define $V_A^*:=\{M\in Mat_n(\C)~|~ M(V_A)\subseteq V_A\}$.
\end{definition}

\begin{remark} \label{04.4.I}
If the minimum polynomial of $A$, $m_A(X)$, has degree $p$, then $dim(V_A)=p$. Also, $V_A^*=V_A\cong \C[X]/m_A(X)$.
\end{remark}

\begin{proof} $dim(V_A)=p$ is trivial.
Now notice that $V_A\subseteq V_A^*$. Suppose $V_A\subsetneq V_A^*$.
A $p$-dimensional vector space can only have a $p$-dimensional set of linear endomorphisms. So, since
$V_A\subsetneq V_A^*$,
there must be $M,M'\in Mat_n(\C)$ such that $M\not = M'$, but
$M=M'$ if restricted to $V_A$, i.e. $M-M'$ is the zero map on $V_A$.
Thus, if one substitutes $A\in V_A$ in the map $M-M'$, it results in the zero matrix. Thus $(M-M')A=0$ but since
$A$ is invertible, $M-M'=0$. Contradiction, as $M\not = M'$. Thus $V_A= V_A^*$, both having $\{I,A,\ldots, A^{p-1}\}$ as 
a basis. 
Now, note that $V_A$ is a commutative ring, and that the 
morphism $\sigma:\C[X]\lp V_A$ given by $\sigma(X)=A$ has $m_A(X)$ in its kernel.
Since $dim(\C[X]/m_A(X))=dim(V_A)$ as $\C$-vector spaces we have an isomorphism between $\C[X]/m_A(X)$ and $V_A$.
\end{proof}

Now let the minimum polynomial of
$A\in GL_n(\C)$ have degree $p$,
and $A^p=\sum_{i=1}^p c_iA^{p-i}$, or equivalently, $I=\sum_{i=1}^p c_iA^{-i}$.
Now one can describe the map $A:V_A\lp V_A$ as a linear map with respect to the basis
$A^{-1},\ldots, A^{-p}$. First, define
$\vec{\lambda}=(\lambda_1,\ldots,\lambda_m)^t$ and $\vec{A}:=(A^{-1},\ldots, A^{-m})$.

\begin{lemma}\label{04.4.N}\label{04.4.O}
If $A\in GL_n(\C)$ and $m(X)$ has degree $p$, then the map $A:V_A\lp V_A$, seen as a linear map with respect to the basis $A^{-1},\ldots, A^{-p}$,
is the matrix $C_m\in Mat_p(\C)$.
In other words, $A<\vec{\lambda},\vec{A}>=<C_m\vec{\lambda},\vec{A}>$
\end{lemma}

\begin{proof}
\[ \begin{array}{rl}
A(\sum_{i=1}^p \lambda_i A^{-i})=&
\sum_{i=2}^{p}\lambda_iA^{-i+1} + \lambda_1I\\
=&\sum_{i=2}^p\lambda_iA^{-i+1}+\lambda_1\sum_{i=1}^n c_iA^{-i}\\
\end{array}\]
thus the matrix of $A:V_A\lp V_A$ w.r.t. the basis $A^{-1},\ldots, A^{-p}$ is the matrix
belonging to the map $\vec{\lambda}\lp C_m\vec{\lambda}$.
\end{proof}

\section{Expressing $A^z$ in $A^{-1},\ldots,A^{-n}$.}

In this section, we assume all maps $z\lp \lambda^z$ to be fixed, and all $\C$-flows $A^z$ to be fixed by fixing the $\C$-flow of one of its Jordan normal forms (doesn't matter
which one of course). 
We will keep on using the notations $\vec{\lambda}:=(\lambda_1,\ldots, \lambda_q)^t$
and $\vec{A}:=(A^{-1},\ldots, A^{-q})$.

\begin{lemma}\label{04.4.P}
$A^z\in V_A$ for all $z\in\C$.
\end{lemma}

\begin{proof}
(i) First, let $B=I+\lambda N\in Mat_n(\C)$ for some $\lambda\in\C$.
Let $W:=\C I +\C N +\ldots + \C N^{n-1}$. Then it is not difficult to see that $W$ is spanned by
$I,B,B^2,\ldots, B^{n-1}$. By lemma \ref{05.4.B} $B^z\in W$, so it follows that $B^z=\sum_{i=0}^{n-1} b_i B^i$ for some $b_i\in\C$.\\
(ii)
 Now, let $J\in Mat_n(\C)$ be a Jordan normal form matrix with $m$ blocks $B_i$ of size $n_i$.
Then $J^z$ is a matrix with blocks $B_i^z$. We hope to find $a_j\in\C$ such that $J^z=\sum_{j=0}^{n-1} a_j J^j$.
This last equation is equivalent to the $m$ equations $B_i^z=\sum_{j=0}^{n-1} a_j B_i^j$.
Using (i), we see that these are $m$ equations in $n$ unknowns. The $i$-th equation has degree
of freedom $n-n_i$, so the total system has degree of freedom $\geq n-\sum_{i=1}^m n_i=0$, so there
exists at least one solution $a_0,\ldots,a_n$ such that $J^z=\sum_{j=0}^n a_j J^j$.\\
(iii) Finally, the theorem follows from the fact that $A^z=\sum_{i=0}^n a_i A^i$
is equivalent to $J^z=\sum_{i=0}^n a_i J^i$ where $J$ is a Jordan normal form of $A$.
\end{proof}

\begin{definition}\label{04.4.II}
Let $A\in GL_n(\C)$, and assume that the minimum polynomial of $A$, $m(X)$, has degree $p$.
Define $\tau: V_A\lp V_{C_m}\subseteq Mat_p(\C)$ as the $\C$-linear ring isomorphism map sending
$A\in V_A^*$ to the map $C_m\in Mat_p(\C)$.
\end{definition}

A short remark: $V_{C_m}$ indeed is isomorphic to $\C[X]/m(X)$ since $m(X)$ is the minimum
polynomial of $C_m$ by lemma \ref{04.4.F}, and notice that $\tau$ is indeed a well-defined isomorphism.

\begin{corollary}\label{2005.06.10}
The map $\tau$ projects maps $M:V_A\lp V_A$ onto their matrix representation w.r.t. the basis $A^{-1},\ldots,A^{-p}$.
In other words, $M<\vec{\lambda},\vec{A}>=<\tau(M)\vec{\lambda},\vec{A}>$.
\end{corollary}

\begin{proof}
Lemma \ref{04.4.O} states that $A<\vec{\lambda},\vec{A}>=<\tau(A)\vec{\lambda},\vec{A}>$.
Since $M\in V_A^*$ one has $M=\sum_{i=0}^{p-1} \alpha_i A^i$ for some $\alpha_i\in\C$.
Thus,
\[ \begin{array}{rl}
M<\vec{\lambda},\vec{A}> =& \sum_{i=0}^{p-1} \alpha_i A^i<\vec{\lambda},\vec{A}>\\
=&\sum_{i=0}^{p-1} \alpha_i <(\tau(A))^i\vec{\lambda},\vec{A}>\\
=<\tau(M)\vec{\lambda},\vec{A}>.
\end{array} \]
\end{proof}

Write $(I0):=(I_n,0_{(p\times n)})$ as the matrix of size $(n+p) \times n$ which has an identity matrix in the first $n$ colums and zero entries anywhere else. Use $M^t$ to denote the transpose of a matrix $M$.

\begin{proposition}\label{04.4.G}
Let the characteristic polynomial $m(X)$ of $A\in GL_n(\C)$ be of degree $p$.\\
(i) Then
$\tau$ has the following form:
There exists $S\in Mat_{p\times n}(\C), W\in Mat_{n\times p}$ satisfying
$WS=I_d$, such that
 for all  $M\in V_A^*$,  $\tau(M)=WMS$.\\
(ii) Let $\tilde{J}:=(I0) J (I0)^t$ be as in lemma \ref{04.4.L}, $V\in GL_n(\C),U\in GL_p(\C)$ such
that $V^{-1}JV=A,U^{-1}\tilde{J}U=C_m$.  Then one can take $W=U^{-1}(I0)V, S=V^{-1}(I0)^tU$.
\end{proposition}

\begin{proof}
$A=V^{-1}JV$ for some $V,J\in GL_n(\C)$ where $J$ is a Jordan normal form of $A$. Furthermore,
$C_m=U^{-1}\tilde{J} U$ for some $U,\tilde{J}\in GL_p(\C)$ where $\tilde{J}$ is a Jordan normal form of
$C_m$. Using remark \ref{04.4.L}, we can conclude that $\tilde{J}=(I0) J (I0)^t$.
Now for each $i\in \Z$:
\[ \begin{array}{rl}
C_m^i=& U^{-1}\tilde{J}^iU = U^{-1}((I0) J^i (I0)^t ) U\\
=& U^{-1}(I0) V A^i V^{-1} (I0)^t U.\\
\end{array} \]
Define $W:=U^{-1} (I0) V, S:=V^{-1} (I0)^t U$.
Then the above equations say $C_q^i=WA^iS$. Also, $WS=U^{-1} (I0) VV^{-1} (I0)^t U=U^{-1}(I0)(I0)^t U=U^{-1}I_pU=U^{-1}U=I_p$.
Now if $M\in V_A$, then
\[ \begin{array}{rl}
\tau(M)=                            &
\tau(\sum_{i=1}^p \lambda_i A^{-i}) \textup{~for~some~}\lambda_i\in\C\\
\sum_{i=1}^p\lambda_i \tau(A^{-i}) \\
\sum_{i=1}^p\lambda_i WA^{-i}S\\
W(\sum_{i=1}^p\lambda_i A^{-i} ) S\\
=WMS.
\end{array}\]
\end{proof}

\begin{proposition}\label{2005.5.26.B}
Let $A\in GL_n(\C)$ having minimum polynomial $m(X):=X^p-c_{p-1}X^{p-1}-\ldots -c_1X-c_0$.
Let $c=(c_{p-1},\ldots,c_0)^t$, and $\mu_i(z):=(C_m^zc)_i$. Then $A^z=\sum_{i=1}^p \mu_i(z) A^{-i}$.
\end{proposition}

\begin{proof}
{\em (i)} $\tau(A^z)=\tau(A)^z$ for each $z\in \C$:
By lemma \ref{04.4.P} we know that for each $z\in \C$, $A^z(V_A)\subseteq V_A$.
Therefore, each map $A^z$ can be seen as a linear map on $V_A$.
By definition \ref{04.4.I} we see that $\tau{(A^i)}=(\tau{A})^i=C_m^i$ for all $i\in\Z$.
Now beware: we cannot immediately state that $\tau(A^z)=\tau(A)^z=C_m^z$ for all $z\in \C$.
For this, consider $M(z):=\tau(A^z)-C_m^z\in Mat_p(H(\C))$, i.e. see $M(z)$ as a matrix of size
$p$ with holomorphic functions as entries.
We are done if we show $M(z)=0$.
Let $J,\tilde{J}$ be as in lemma \ref{04.4.L}.
Now using proposition \ref{04.4.G}:
\[ \begin{array}{rl}
M(z)=& \tau(A^z)-C_m^z\\
=& WA^zS-C_m^z\\
=&U^{-1}(I0)V^{-1} VJ^zV V^{-1}(I0)^tU-U^{-1}\tilde{J}^zU\\
=&U^{-1}(I0)J^z (I0)^tU-U^{-1}\tilde{J}^zU
\end{array}\]
By assumption (lemma \ref{04.4.L}) we have
$(I0)J(I0)^t=\tilde{J}$, thus $M(z)=0$.\\
{\em (ii)} Notice that $I=<c,\vec{A}>$.
Now using corollary \ref{2005.06.10} and {\em (i)}, we see that $A^z=A^z<c,\vec{A}>=<\tau(A^z)c,\vec{A}>=<C_m^zc, \vec{A}>$. This gives $A^z=\sum_{i=1}^p (C_mc)_iA^{-i}$.
\end{proof}

We are now almost able to prove the main theorem, we need just one more lemma:

\begin{lemma}\label{2005.5.25.A}
Let $A\in Mat_n(\C)$, let $m(X)$ be the minimum polynomial of $A$ of degree $p$, and let $f(X)\in \C[X]$ be a nonzero
polynomial of degree $d$.
Then there exists $\tilde{A}\in Mat_{n+d}(\C)$ such that it (1) has minimum polynomial $m(X)f(X)$, (2) the upper left $n\times n$ part of $\tilde{A}^z$ is equal to $A^z$ for every $z\in\C$.
\end{lemma}

\begin{proof}
We will replace (2) by (2') and (2'') where (2') is: ``the upper left $n\times n$ part of $\tilde{A}$ is equal to $A$''
and (2'') is ``the rows $n+1$,\ldots $n+d$ have zeroes below the diagonal''. If we can guarantee (2') and (2''), then (2) will hold. The proof will go in some steps.\\
(i) It is enough to prove (1),(2') and (2'') for $A$ on Jordan normal form:
Let $A=T^{-1}JT$ where $J$ is a Jordan normal form of $A$. Let $\tilde{T}\in Mat_{n+p}(\C)$ be the canonical extension of $T$: the upper left $n\times n$ part equals $T$, and the rest of the coefficients equal the coefficients of an
identity matrix. Let $\tilde{J}$ be satisfying (1), (2') and (2''). Then one can take $\tilde{A}:=\tilde{T}\tilde{J}\tilde{T}^{-1}$.\\
(ii) It is enough to prove the theorem for $f(X)=X-a$ of degree one: the full theorem follows by induction.
We will split into the case that $a$ is an eigenvalue of $A$, and the case that it is not.\\
(iii) Suppose $a$ is not an eigenvalue of $A$: then define
\[ \tilde{A}:= \left( \begin{array}{cc}
A & 0\\
0 & a
\end{array} \right) .\]
It is easy to check that this matrix satisfies the criteria.\\
(iv) Suppose $a$ is an eigenvalue of $A$. We assume $A$ on Jordan normal form, and let $B_1,\ldots,B_m$ be the
blocks of $A$. We assume that $B_m:=B(a,n_m)$ is the largest block having $a$ as eigenvalue (or one of the largest, if there are more).
Now define the $(n+1)\times(n+1)$ matrix
\[ \tilde{A}:= \left( \begin{array}{cccc}
B_1 & \ldots&0 &0\\
\vdots&\ddots& & \vdots\\
0&\ldots&B_{m-1}&0\\
0&\ldots&0&\tilde{B}_m\\
\end{array} \right) \]
where $\tilde{B}_m:=B(a,n_m+1)$, i.e. a block with the same eigenvalue as $B_m$ but one size larger.
We leave it to the reader to verify the following three steps:
The minimum polynomial of $\tilde{A}$ will now be (1) of degree one more than the minimum polynomial of $A$; (2) divisible by the minimum polynomial $m(X)$ of A;
(3) divisible by the minimum polynomial of $\tilde{B}_m$: $(X-a)^{n_m+1}$.
Since $(X-a)^{n_m}$ divides $m(X)$ but $(X-a)^{n_m+1}$ does not, the minimum polynomial of $A$ is $m(X)(X-a)$.\\
\end{proof}

And now we have the proof of the main theorem:

\begin{proof}{\em (of theorem \ref{2005.5.26.C}:)}
Suppose the minimum polynomial of $A$ has degree $q\leq p$.
By lemma \ref{2005.5.25.A} we can find a matrix $\tilde{A}\in Mat_{n+p-q}(\C)$ such that it (1) has minimum polynomial $Q(X)$, (2) the upper left $n\times n$ part of $\tilde{A}^z$ is equal to $A^z$ for every $i\in \Z$.

By proposition \ref{04.4.G} we have $\tilde{A}^z=\sum_{i=1}^p \mu_i(z) \tilde{A}^{-i}$.
Now restricting to the upper left $n\times n$ coefficients we have the equality
$A^z=\sum_{i=1}^p \mu_i(z) A^{-i}$.
\end{proof}

\section{Practical computation of the analytic functions}

When one would like to compute the analytic functions $\mu_i(z)$ in a practical situation,
given a relation $Q(A):=A^p-c_{p-1}A^{p-1}-\ldots -c_1A-c_0=0$, one may now use theorem
\ref{2005.5.26.C} as a basis for a more efficient computation.
Instead of computing $C_Q^z$, it may be more easy to know which form the analytic functions have, and compute it directly. Since $C_Q$ is a conjugation of a Jordan-normal form $J$, we know that $C_Q^z$ is a conjugation of $J^z$.
The Jordan normal form $J$ of $C_Q$ can be computed by finding the roots with multiplicity of $Q(X)$: $\lambda_1,\ldots, \lambda_m$ with multiplicity $n_1,\ldots, n_m$. Now one knows which analytic functions occur in
$J^z$: let $B_1,\ldots,B_m$ be the blocks of $J$, then block $B_i:=\sum_{j=0}^{p-1}
g_j(z)\lambda^{z-j}N^j$ where $g_j(z):=z(z-1)\ldots (z-j+1)$ if $j>0$ and $g_0(z)=1$ (See definition \ref{05.4.A} and lemma \ref{05.4.B}).
Define the functions $f_1(z),\ldots,f_p(z)$ as the $p$ functions appearing as coefficients in $J^z$: $g_j(z)\lambda_i^{z-j}$ where $1\leq i \leq m$, $0\leq j \leq n_{i}-1$,
one knows that $C_Q^z$ has entries which are linear combinations of the $f_i$, and thus so does $C_Q^zc$.
So, in a practical situation it may be easier to compute $e_{ij}$ such that
\[ A^z=\sum_{i=1}^{p} \Big(  \sum_{j=1}^p e_{ij} f_j(z)                \Big)      A^{-i} .\]
In fact, we claim that the $p\times p$ matrix $(e_{ij})$ can be computed as the inverse of the $p\times p$ matrix $(f_j(-i))$:

\begin{lemma}\label{2005.5.30.A}
The matrix $B:=(f_i(-j))$ (where the rows $i$ and the colums $j$ run from 1 to $p$) is invertible, and
\[A^z=\sum_{i=1}^p \Big( \sum_{j=1}^p (B^{-1})_{ij} f_j(z)                \Big)      A^{-i} .\]
\end{lemma}

\begin{proof}
One wants to find $e_{ij}$ such that
\[A^z=\sum_{i=1}^p \Big( \sum_{j=1}^p e_{ij} f_j(z)                \Big)      A^{-i} .\]
Now notice that this is equivalent to
\[
A^z =
\left( \begin{array}{c} f_1(z),\ldots,f_p(z) \end{array} \right)
\left(
\begin{array}{ccc}
e_{11}&\ldots&e_{p1}\\
\vdots&&\vdots\\
e_{1p}&&e_{pp}
\end{array} \right)
\left( \begin{array}{c}
A^{-1}\\
\vdots\\
A^{-p}
\end{array} \right).
\]
Substituting $z=-1,\ldots,-p$ one gets
\[
\left(
\begin{array}{c}
A^{-1}\\
\vdots\\
A^{-p}\\
\end{array}
\right)
=
\left( \begin{array}{ccc}
f_1(-1)&\ldots &f_p(-1)\\
\vdots&&\vdots\\
f_1(-p)&\ldots&f_p(-p)\\
\end{array} \right)
\left(
\begin{array}{ccc}
e_{11}&\ldots&e_{p1}\\
\vdots&&\vdots\\
e_{1p}&&e_{pp}
\end{array} \right)
\left( \begin{array}{c}
A^{-1}\\
\vdots\\
A^{-p}
\end{array} \right).
\]
In case $(f_i(-j))_{ij}$ is an invertible matrix, then the $e_{ij}$ are unique and thus must be the solution that exists according to theorem \ref{2005.5.26.C}. 
Such a matrix $(f_i(-j))_{ij}$ is called a {\em generlized Vandermonde matrix} and it is proven in \cite{Sob02} theorem 1 that these matrices have a determinant
which is nonzero. Thus , $f_j(-i)$ (rows $i$, colums $j$) is the inverse of $e_{ij}$ (rows $j$, colums $i$), which gives the result.
\end{proof}

{\bf Acknowledgements}\\
The author would like to express his gratitude to prof. Jean-Philippe Furter, prof. Arno van den Essen and dr.ing. A.J.E.M.Janssen for
some priceless (email) discussion and corrective work.

\end{document}